\documentclass[12pt, reqno]{amsart}
\usepackage{amsmath, amsthm, amscd, amsfonts, amssymb, graphicx, color}
\usepackage[ plainpages]{hyperref}

\newtheorem{theorem}{Theorem}[section]
\newtheorem{lemma}[theorem]{Lemma}
\newtheorem{proposition}[theorem]{Proposition}
\newtheorem{corollary}[theorem]{Corollary}

\theoremstyle{definition}

\theoremstyle{remark}
\newtheorem{remark}[theorem]{Remark}

\newcommand{\tr}{{\rm{tr}}}
\usepackage[all]{xy}
\numberwithin{equation}{section}

\newtheorem*{theorem*}{Theorem}

\begin{document}
\title[Means refinements via convexity  ]{Means refinements via convexity}
\author[M. Sababheh]{M. Sababheh}\address{Department of Basic Sciences, Princess Sumaya University For Technology, Al Jubaiha, Amman 11941, Jordan.}\email{\textcolor[rgb]{0.00,0.00,0.84}{sababheh@psut.edu.jo, sababheh@yahoo.com}}


\maketitle
\begin{abstract}
The main goal of this article is to find the exact difference between a convex function and its secant, as a limit of positive quantities. This idea will be expressed as a convex inequality that leads to refinements and reversals of well established inequalities treating different means. The significance of these inequalities is to write one inequality that brings together and refine almost all known inequalities treating the arithmetic, geometric, harmonic and Heinz means, for numbers and operators.
\end{abstract}
\section{introduction}
Convex functions and their inequalities have played a major role in the study of various topics in Mathematics; including applied Mathematics, Mathematical Analysis and Mathematical Physics. Means and their comparison is indeed an important application of convexity.\\
Recall that a function $f:\mathbb{I}\to\mathbb{R}$, defined on a real interval $\mathbb{I}$, is said to be convex if $f(\alpha x_1+\beta x_2)\leq \alpha f(x_1)+\beta f(x_2)$, when $x_1,x_2\in\mathbb{I}$ and $\alpha,\beta\geq 0$ satisfying $\alpha+\beta=1.$ On the other hand, $f:\mathbb{I}\to\mathbb{R}^+$ is said to be log-convex if $g(x)=\log f(x)$ is convex, or equivalently if $ f(\alpha x_1+\beta x_2)\leq f^{\alpha}(x_1)f^{\beta}(x_2)$ for the above parameters.

Speaking of means, the comparison between the weighted arithmetic, geometric and harmonic means is an immediate consequence of convexity or log-convexity of the functions $x\nabla_ty=(1-t)x+ty, x\#_ty=x^{1-t}y^t$ and $x!_ty=((1-t)x^{-1}+ty^{-1})^{-1}, x,y>0,$ defined for $0\leq t\leq 1$. Adopting these notations, we drop $t$ when $t=\frac{1}{2}.$\\
Convexity of the function $f(t)=x\#_ty$ implies the well known Young's inequality $x\#_t\leq x\nabla_ty.$ On the other hand, convexity of the function $g(t)=x!_ty$ implies the arithmetic-harmonic mean inequality $x!_ty\leq x\nabla_ty$, while log-convexity of $g$ implies the geometric-harmonic mean inequality $x!_ty\leq x\#_ty.$

These inequalities, though very simple, have some significant applications. For example, the above Young's inequality implies the celebrated Holder's inequality $\|fg\|_1\leq \|f\|_p\|g\|_q$ for $f\in L^p(X)$ and $g\in L^q(X)$, for the conjugate exponents $p,q$, where $X$ is some measure space.

Among the most interesting applications of the above mean inequalities is the possible comparison between operators acting on a finite dimensional Hilbert space $H$. In the sequel, $\mathbb{M}_n$ will denote the space of operators acting on an $n-$deimentional Hilbert space $H$, $\mathbb{M}_n^+$ will denotes the cone of semi positive operators in $\mathbb{M}_n$ while $\mathbb{M}_n^{++}$ will denotes the cone of strictly positive operators in $\mathbb{M}_n.$ Then the above numerical inequalities have their operator versions such as $A\#_t B\leq A\#_t B$, where $A,B\in\mathbb{M}_n^{++}, A\nabla_t B=(1-t)A+t B$ and $A\#_tB=A^{\frac{1}{2}}\left(A^{-\frac{1}{2}}BA^{-\frac{1}{2}}\right)^{t}A^{\frac{1}{2}}.$ In this context, we say that $A\leq B$ for two self-adjoint operators $A$ and $B$ if $B-A\in\mathbb{M}_n^+.$

Obtaining the operator versions from the corresponding numerical versions can be done in different approaches, among which is the application of the following lemma \cite{furtu}.
\begin{lemma}\label{monotone}
Let $X\in\mathcal{M}_n$ be self-adjoint and let $f$ and $g$ be continuous real valued functions such that $f(t)\geq g(t)$ for all $t\in{\text{Sp}}(X),$ the spectrum of $X$. Then $f(X)\geq g(X).$
\end{lemma}
Recent studies of the topic have investigated possible refinements of the above inequalities, where adding a positive term to the left side becomes possible. This idea has been treated in \cite{{omarkittaneh},{kittanehmanasreh},{liao},{kittmanas2},{saboam},{sabjmaa},{sab_conv},{zhao},{zuo}}, where not only refinements have been investigated, but reversed versions and much more have been discussed.

Keeping our paper concise, we will not go through  the exact results done in the above references now, however we will comment later how the results in this paper generalize almost all results in these references, regarding the refinements and the reverses of the above mean inequalities.

The main goal of this article is to avoid dealing with the specific means, and to treat a general convexity argument that leads to these refinements. In particular, we prove that for certain positive quantities $A_j(\nu)\Delta_jf(\nu;a,b),$ we have
\begin{eqnarray*}
 f\left((1-\nu)a+\nu b\right)+\sum_{j=1}^{N}A_j(\nu)\Delta_jf(\nu;a,b)\leq (1-\nu)f(a)+\nu f(b), N\in\mathbb{N},
\end{eqnarray*}
for the convex function $f:[a,b]\to\mathbb{R}$. This provides $N$ refining terms of the inequality $f\left((1-\nu)a+\nu b\right)\leq (1-\nu)f(a)+\nu f(b)$, which follows from convexity of $f$. Furthermore, we prove a reversed version and we prove that as $N\to\infty$ the above inequality becomes an equality. As a natural consequence, we obtain some refinements and reverses for log-convex functions.

As we will see, the above inequality and its consequences happen to be  generalizations that imply almost all inequalities in the references \cite{omarkittaneh,kittanehmanasreh,liao,saboam,sabjmaa,sab_conv,zhao,zuo}. This is our main motivation behind this work; to find a formula that implies and generalizes all other formulae and hence, to enhance our understanding of these inequalities.

We remark that the proof of the first main result in this work is inspired by our recent work in \cite{sabjmaa}.

\section{main results}
For the rest of the paper, the following notations will be adopted. For $0\leq \nu\leq 1$ and $j\in\mathbb{N}$, let
\begin{equation}\label{k_j_definition}
\left\{\begin{array}{cc}k_j(\nu)=[2^{j-1}\nu], r_j(\nu)=[2^{j}\nu]\;{\text{and}}\\
A_j(\nu)=(-1)^{r_j(\nu)}2^{j-1}\nu+(-1)^{r_j(\nu)+1}\left[\frac{r_j(\nu)+1}{2}\right]\end{array}\right..
\end{equation}
Moreover, if $f:[a,b]\to\mathbb{R}$ is any function, define
\begin{eqnarray}
\nonumber \Delta_jf(\nu;a,b)&=&f\left(\left(1-\frac{k_j(\nu)}{2^{j-1}}\right)a+\frac{k_j(\nu)}{2^{j-1}}b\right)+
f\left(\left(1-\frac{k_j(\nu)+1}{2^{j-1}}\right)a+\frac{k_j(\nu)+1}{2^{j-1}}b\right)\\
\label{definition_Delta}&-&2f\left(\left(1-\frac{2k_j(\nu)+1}{2^{j}}\right)a+\frac{2k_j(\nu)+1}{2^{j}}b\right), 0\leq \nu\leq 1.
\end{eqnarray}
\subsection{Convex functions} We discuss first the inequalities that govern convex functions, then we apply these inequalities to log-convex functions.
\begin{lemma}\label{first_lemma}
If $f:[a,b]\to\mathbb{R}$ is convex, then $\Delta_jf(\nu;a,b)\geq 0$ for $j\in\mathbb{N}$ and $0\leq \nu\leq 1.$
\end{lemma}
\begin{proof}
Letting $x_j(\nu)=\left(1-\frac{k_j(\nu)}{2^{j-1}}\right)a+\frac{k_j(\nu)}{2^{j-1}}b, y_j(\nu)=\left(1-\frac{k_j(\nu)+1}{2^{j-1}}\right)a+\frac{k_j(\nu)+1}{2^{j-1}}b$
and $z_j(\nu)=\left(1-\frac{2k_j(\nu)+1}{2^{j}}\right)a+\frac{2k_j(\nu)+1}{2^{j}}b,$ it is easy that
$z_j(\nu)=\frac{x_j(\nu)+y_j(\nu)}{2}.$ The $\Delta_{j}f(\nu;a,b)=f(x_j(\nu))+f(y_j(\nu))-2f(z_j(\nu))\geq 0,$ by convexity of $f$.
\end{proof}
\begin{remark}
When $f:[a,b]\to\mathbb{R}$, we adopt the convention that $f(x)=0$ for $x\not\in[a,b].$ This convention will be needed, for example, in the next lemma, when $N=1$ and $\nu=1.$
\end{remark}
\begin{lemma}\label{lemma_exact_difference}
Let $f:[0,1]\to\mathbb{R}$ be a function and let $N\in\mathbb{N}$. Then
\begin{eqnarray}
\nonumber(1-\nu)f(0)&+&\nu f(1)-\sum_{j=1}^{N}A_{j}(\nu)\Delta_{j}f(\nu;0,1)\\
\nonumber&=&\left([2^{N}\nu]+1-2^{N}\nu\right)f\left(\frac{[2^N\nu]}{2^N}\right)+\left(2^{N}\nu-[2^{N}\nu]\right)f\left(\frac{[2^N\nu]+1}{2^N}\right).\\
\label{exact_difference}&&
\end{eqnarray}
\end{lemma}
\begin{proof}
We proceed by induction on $N$.\\
 When $N=1$ and $0\leq \nu<\frac{1}{2},$  $r_1(\nu)=0$ and $k_1(\nu)=0$. Hence $A_1(\nu)=\nu$ and
$\Delta_1f(\nu;a,b)=f(a)+f(b)-2f\left(\frac{a+b}{2}\right).$ Then direct computations show the result.\\
Now if $\frac{1}{2}\leq\nu<1$, then $r_1(\nu)=1$ and $k_1(\nu)=0$, hence $A_1(\nu)=1-\nu$ and $\Delta_1f(\nu;a,b)=f(a)+f(b)-2f\left(\frac{a+b}{2}\right).$ Again, direct computations show the result.\\
When $\nu=1,$ the result follows immediately.\\
Now assume that (\ref{exact_difference}) is true for some $N\in\mathbb{N}$. We assert its truth for $N+1.$ Notice that, using the inductive step,
\begin{eqnarray}
\nonumber (1-\nu)f(0)&+&\nu f(1)-\sum_{j=1}^{N+1}A_{j}(\nu)\Delta_{j}f(\nu;0,1)\\
\nonumber&=&(1-\nu)f(0)+\nu f(1)-\sum_{j=1}^{N}A_{j}(\nu)\Delta_{j}f(\nu;0,1)-A_{N+1}(\nu)\Delta_{N+1}f(\nu;0,1)\\
\nonumber&=&\left([2^{N}\nu]+1-2^{N}\nu\right)f\left(\frac{[2^N\nu]}{2^N}\right)+\left(2^{N}\nu-[2^{N}\nu]\right)f\left(\frac{[2^N\nu]+1}{2^N}\right)\\
\nonumber&-&\left((-1)^{[2^{N+1}\nu]}2^{N}\nu+(-1)^{[2^{N+1}\nu]+1}\left[\frac{[2^{N+1}\nu]+1}{2}\right]\right)\times\\
\label{needed_in_induction}&\times&\left(f\left(\frac{[2^N\nu]}{2^N}\right)+f\left(\frac{[2^N\nu]+1}{2^N}\right)-2f\left(\frac{2[2^N\nu]+1}{2^{N+1}}\right)\right).
\end{eqnarray}
Now we treat two cases.\\
{\bf{Case I}} If $[2^{N+1}\nu]$ is odd, then we easily see that $[2^{N}\nu]=\frac{[2^{N+1}\nu]-1}{2}.$ Therefore,
$$f\left(\frac{[2^N\nu]+1}{2^N}\right)=f\left(\frac{[2^{N+1}\nu]+1}{2^{N+1}}\right)\;{\text{and}}
\;f\left(\frac{2[2^N\nu]+1}{2^{N+1}}\right)=f\left(\frac{[2^{N+1}\nu]}{2^{N+1}}\right).$$ Substituting these values in (\ref{needed_in_induction}) and simplifying imply
\begin{eqnarray}
\nonumber&& (1-\nu)f(0)+\nu f(1)-\sum_{j=1}^{N+1}A_{j}(\nu)\Delta_{j}f(\nu;0,1)\\
\nonumber&=&\left(2^{N+1}\nu-[2^{N+1}\nu]\right)f\left(\frac{[2^{N+1}\nu]+1}{2^{N+1}}\right)+
\left([2^{N+1}\nu]+1-2^{N+1}\nu\right)f\left(\frac{[2^{N+1}\nu]}{2^{N+1}}\right),
\end{eqnarray}
which completes the proof, when $[2^{N+1}\nu]$ is odd.\\
{\bf{Case II}} If $[2^{N+1}\nu]$ is even, then $2[2^{N}\nu]=[2^{N+1}\nu]$ and
$$f\left(\frac{[2^N\nu]}{2^N}\right)=f\left(\frac{[2^{N+1}\nu]}{2^{N+1}}\right)\;{\text{and}}
\;f\left(\frac{2[2^N\nu]+1}{2^{N+1}}\right)=f\left(\frac{[2^{N+1}\nu]+1}{2^{N+1}}\right).$$
Substituting these values in (\ref{needed_in_induction}) and simplifying imply
\begin{eqnarray}
\nonumber&& (1-\nu)f(0)+\nu f(1)-\sum_{j=1}^{N+1}A_{j}(\nu)\Delta_{j}f(\nu;0,1)\\
\nonumber&=&\left(2^{N+1}\nu-[2^{N+1}\nu]\right)f\left(\frac{[2^{N+1}\nu]+1}{2^{N+1}}\right)+
\left([2^{N+1}\nu]+1-2^{N+1}\nu\right)f\left(\frac{[2^{N+1}\nu]}{2^{N+1}}\right).
\end{eqnarray}
This completes the proof.
\end{proof}

\begin{corollary}\label{corollary_main_result_[0,1]}
Let $f:[0,1]\to\mathbb{R}$ be convex and let $N\in\mathbb{N}$. Then
\begin{equation}
f(\nu)+\sum_{j=1}^{N}A_j(\nu)\Delta_jf(\nu;0,1)\leq (1-\nu)f(0)+\nu f(1).
\end{equation}
\end{corollary}
\begin{proof}
From Lemma \ref{lemma_exact_difference} and convexity of $f$, we have
\begin{eqnarray*}
\nonumber(1-\nu)f(0)&+&\nu f(1)-\sum_{j=1}^{N}A_{j}(\nu)\Delta_{j}f(\nu;0,1)\\
\nonumber&=&\left([2^{N}\nu]+1-2^{N}\nu\right)f\left(\frac{[2^N\nu]}{2^N}\right)+\left(2^{N}\nu-[2^{N}\nu]\right)f\left(\frac{[2^N\nu]+1}{2^N}\right)\\
&\geq&f\left(\left([2^{N}\nu]+1-2^{N}\nu\right)\frac{[2^N\nu]}{2^N}+\left(2^{N}\nu-[2^{N}\nu]\right)\frac{[2^N\nu]+1}{2^N}\right)\\
&=&f(\nu).
\end{eqnarray*}
This completes the proof.
\end{proof}
Now our first main result in its general form can be stated as follows.
\begin{theorem}\label{first_main_theorem}
Let $f:[a,b]\to\mathbb{R}$ be convex. Then for each $N\in\mathbb{N}$ and $0\leq \nu\leq 1,$ we have
\begin{eqnarray}
\label{first_main_inequality} f\left((1-\nu)a+\nu b\right)+\sum_{j=1}^{N}A_j(\nu)\Delta_jf(\nu;a,b)\leq (1-\nu)f(a)+\nu f(b).
\end{eqnarray}
\end{theorem}
\begin{proof}
For the given $f$, define $g:[0,1]\to \mathbb{R}$ by $g(x)=f((1-x)a+xb).$ Then $g$ is convex on $[0,1].$ Applying Corollary \ref{corollary_main_result_[0,1]} on the function $g$ implies the result.
\end{proof}
\begin{remark}
We remark that a negative version of the above theorem has been recently shown in \cite{sabmia}. Namely, it was proved
\begin{align}\label{main_introduction}
\nonumber(1+\nu)f(a)-\nu f(b)&+\sum_{j=1}^{N}2^{j}\nu\left[  \frac{f(a)+f\left(\frac{(2^{j-1}-1)a+b}{2^{j-1}}\right)}{2}-f\left(\frac{(2^j-1)a+b}{2^j}\right)\right]\\
&\leq f\left((1+\nu)a-\nu b\right), \nu\geq 0, a<b,
\end{align}
for the convex function $f:\mathbb{R}\to\mathbb{R}.$ However, the method of proof is considerably easier than the above proofs and the applications are different.
\end{remark}
Our next step is to prove a reversed version of (\ref{first_main_inequality}).
\begin{theorem}\label{first_reversed_theorem}
Let $f:[a,b]\to\mathbb{R}$ be convex and let $N\in\mathbb{N}$. Then for $0\leq \nu\leq \frac{1}{2},$ we have
\begin{eqnarray*}
f\left((1-\nu)a+\nu b\right)&+&(1-A_1(\nu))\Delta_1f(\nu;a,b)\\
&\geq& (1-\nu)f(a)+\nu f(b)+\sum_{j=1}^{N}A_j(1-2\nu)\Delta_jf\left(1-2\nu;\frac{a+b}{2},b\right).
\end{eqnarray*}
On the other hand, if $\frac{1}{2}\leq\nu\leq 1,$ we have
\begin{eqnarray*}
f\left((1-\nu)a+\nu b\right)&+&(1-A_1(\nu))\Delta_1f(\nu;a,b)\\
&\geq& (1-\nu)f(a)+\nu f(b)+\sum_{j=1}^{N}A_j(2-2\nu)\Delta_jf\left(2-2\nu;a,\frac{a+b}{2}\right).
\end{eqnarray*}
\end{theorem}
\begin{proof}
For $0\leq \nu\leq \frac{1}{2},$ we have
\begin{eqnarray*}
&&f\left((1-\nu)a+\nu b\right)+(1-A_1(\nu))\Delta_1f(\nu;a,b)-\left((1-\nu)f(a)+\nu f(b)\right)\\
&=&2\nu f\left(\frac{a+b}{2}\right)+(1-2\nu)f(b)+f\left((1-\nu)a+\nu b\right)-2f\left(\frac{a+b}{2}\right)\\
&\geq&\sum_{j=1}^{N}A_j(1-2\nu)\Delta_jf\left(1-2\nu;\frac{a+b}{2},b\right)+f\left(2\nu\frac{a+b}{2}+(1-2\nu)b\right)\\
&&+f\left((1-\nu)a+\nu b\right)-2f\left(\frac{a+b}{2}\right)\\
&=&\sum_{j=1}^{N}A_j(1-2\nu)\Delta_jf\left(1-2\nu;\frac{a+b}{2},b\right)+f\left(\nu a+(1-\nu) b\right)\\
&&+f\left((1-\nu)a+\nu b\right)-2f\left(\frac{a+b}{2}\right)\\
&\geq&\sum_{j=1}^{N}A_j(1-2\nu)\Delta_jf\left(1-2\nu;\frac{a+b}{2},b\right),
\end{eqnarray*}
where the last line follows from convexity of $f$, where one has
\begin{eqnarray*}
&&f\left(\nu a+(1-\nu) b\right)+f\left((1-\nu)a+\nu b\right)\\
&&\geq 2f\left(\frac{\nu a+(1-\nu) b+(1-\nu)a+\nu b}{2}\right)=2f\left(\frac{a+b}{2}\right).
\end{eqnarray*}
This completes the proof for $0\leq \nu\leq \frac{1}{2}.$ Similar computations imply the desired inequality for $\frac{1}{2}\leq\nu\leq 1.$
\end{proof}

In fact, the above reversed version turns out to be equivalent to convexity.
\begin{proposition}
Let $f:\mathbb{I}\to\mathbb{R}$ be a function defined on the interval $\mathbb{I}$. Assume that for all $a<b$ in $\mathbb{I}$ and all $0\leq\nu\leq 1$, we have
\begin{equation}\label{needed_converse_convexity}
f\left((1-\nu)a+\nu b\right)+(1-A_1(\nu))\Delta_1f(\nu;a,b)\geq (1-\nu)f(a)+\nu f(b),
\end{equation}
 then $f$ is convex on $\mathbb{I}.$
\end{proposition}
\begin{proof}
Observe that when $0\leq\nu\leq\frac{1}{2},$ (\ref{needed_converse_convexity}) is equivalent to
$$f\left((1-\nu)a+\nu b\right)+(1-\nu)\left(f(a)+f(b)-2f\left(\frac{a+b}{2}\right)\right)\geq (1-\nu)f(a)+\nu f(b),$$ or
\begin{equation}\label{first_needed_lemma_converse_convexity}
f\left(\frac{a+b}{2}\right)\leq \frac{1}{2-2\nu}f\left((1-\nu)a+\nu b\right)+\frac{1-2\nu}{2-2\nu}f(b).
\end{equation}
On the other hand, if $\frac{1}{2}\leq\nu\leq 1,$ (\ref{needed_converse_convexity}) is equivalent to
\begin{equation}\label{second_needed_lemma_converse_convexity}
f\left(\frac{a+b}{2}\right)\leq \frac{2\nu-1}{2\nu}f(a)+\frac{1}{2\nu}f\left((1-\nu)a+\nu b\right).
\end{equation}
Let $x_1<x_2\in\mathbb{I}$ and let $0<\lambda<1.$ We assert that $f((1-\lambda)x_1+\lambda x_2)\leq (1-\lambda)f(x_1)+\lambda f(x_2).$\\
If $0<\lambda\leq \frac{1}{2},$ let $$\nu=\frac{1-2\lambda}{2(1-\lambda)}, a=(2-2\lambda)x_1+(2\lambda-1)x_2\;{\text{and}}\;b=x_2.$$ Then one can easily check that when $0<\lambda\leq\frac{1}{2},$ we have $0<\nu\leq \frac{1}{2}$ and $a<b$. With these choices, we have
$$\frac{a+b}{2}=(1-\lambda)x_1+\lambda x_2\;{\text{and}}\;(1-\nu)a+\nu b=x_1.$$ Substituting these quantities in (\ref{first_needed_lemma_converse_convexity}) implies $f((1-\lambda)x_1+\lambda x_2)\leq (1-\lambda)f(x_1)+\lambda f(x_2).$ This proves the desired inequality for $0<\lambda\leq\frac{1}{2}.$\\
Now if $\frac{1}{2}\leq\lambda<1,$ let $$\nu=\frac{1}{2\lambda}, a=x_1\;{\text{and}}\;b=(1-2\lambda)x_1+2\lambda x_2.$$ With these choices, we have $\frac{1}{2}<\nu\leq 1$ and $a<b$. Now substituting these quantities in (\ref{second_needed_lemma_converse_convexity}) implies the desired inequality for $\frac{1}{2}\leq\lambda<1.$ This completes the proof.
\end{proof}

As for the geometric meaning of these refinements, it turns out we are dealing with the interpolation of the  function $f$ over the dyadic partition.
\begin{proposition}\label{equality_dyadic}
Let $f:[0,1]\to\mathbb{R}$ be any function, and let $N\in\mathbb{N}$. Then, if $\nu_i=\frac{i}{2^{N}}$ for some $i=0,1,\cdots,2^N,$ we have
\begin{equation}
f(\nu_i)+\sum_{j=1}^{N}A_j(\nu_i)\Delta_jf(\nu_i;0,1)= (1-\nu_i)f(0)+\nu_i f(1).
\end{equation}
\end{proposition}
\begin{proof}
Observe that when $\nu_i=\frac{i}{2^N},$ we have $[2^{N}\nu_i]=2^{N}\nu_i=i.$ From Lemma \ref{exact_difference}, we have
\begin{eqnarray*}
(1-\nu_i)f(0)&+&\nu_i f(1)-\sum_{j=1}^{N}A_j(\nu_i)\Delta_jf(\nu_i;0,1)\\
&=&\left([2^{N}\nu_i]+1-2^{N}\nu_i\right)f\left(\frac{[2^N\nu_i]}{2^N}\right)+\left(2^{N}\nu_i-[2^{N}\nu_i]\right)f\left(\frac{[2^N\nu_i]+1}{2^N}\right)\\
&=&f\left(\frac{i}{2^N}\right)=f(\nu_i).
\end{eqnarray*}
This completes the proof.
\end{proof}

\begin{proposition}\label{limit}
Let $f:[0,1]\to\mathbb{R}$ be a given function. If $f$ is continuous, then
\begin{equation}
f(\nu)+\lim_{N\to\infty}\sum_{j=1}^{N}A_j(\nu)\Delta_jf(\nu;0,1)= (1-\nu)f(0)+\nu f(1),
\end{equation}
uniformly in $\nu\in[0,1].$
\end{proposition}
\begin{proof}
Let $N\in\mathbb{N}$ and define the function $$g_N(\nu)=\sum_{j=1}^{N}A_j(\nu)\Delta_jf(\nu;0,1).$$ From Proposition \ref{equality_dyadic}, we have
$g(\nu_i)=(1-\nu_i)f(0)+\nu_i f(1)-f(\nu_i),$ when $\nu_i=\frac{i}{2^N}$ for some $i=1,\cdots,2^N.$ Noting the definitions of $A_j$ and $\Delta_jf$, one can easily see that $g_N$ is linear on each dyadic interval $I_i:=\left[\frac{i}{2^N},\frac{i+1}{2^N}\right], i=0,\cdots,2^N-1.$ Now since $g_N$ is  linear on $I_i$ and $g_N$ coincides with the continuous function $h(\nu):=(1-\nu)f(0)+\nu f(1)-f(\nu),$ it follows that $g_N$ is the linear interpolation of $h$ at the dyadic partition of $[0,1].$ Since $f$ is continuous, it follows that $g_N\to h$ uniformly, completing the proof.
\end{proof}

\subsection{Log-Convex function}
The proof of the following result follows from Theorem \ref{first_main_theorem} on replacing $f$ by $\log f.$
\begin{corollary}\label{corollar_log_convex_first}
Let $f:[a,b]\to (0,\infty)$ be log-convex. Then for $0\leq \nu\leq 1$ and $N\in\mathbb{N}$, we have
\begin{equation}
f\left((1-\nu)a+\nu b\right)\prod_{j=1}^{N}\left(\frac{f(x_j(\nu))f(y_j(\nu))}{f^2(z_j(\nu))}\right)^{A_j(\nu)}\leq f^{1-\nu}(a)f^{\nu}(b),
\end{equation}
where $x_j(\nu), y_j(\nu)$ and $z_j(\nu)$ are as in the proof of Lemma \ref{first_lemma}.
\end{corollary}

On the other hand, applying Theorem \ref{first_reversed_theorem} implies the following.
\begin{corollary}\label{corollar_log_convex_second}
Let $f:[a,b]\to (0,\infty)$ be log-convex. Then for $0\leq \nu\leq \frac{1}{2}$ and $N\in\mathbb{N}$, we have
\begin{equation}
f\left((1-\nu)a+\nu b\right)\left(\frac{f(a)f(b)}{f^2(\frac{a+b}{2})}\right)^{1-A_1(\nu)}
\geq f^{1-\nu}(a)f^{\nu}(b)\prod_{j=1}^{N}\left(\frac{f(t_j(\nu))f(u_j(\nu))}{f^2(w_j(\nu))}\right)^{A_j(1-2\nu)},
\end{equation}
where $t_j(\nu), u_j(\nu)$ and $w_j(\nu)$ are obtained from the above $x_j(\nu), y_j(\nu)$ and $z_j(\nu)$ on replacing $(\nu,a,b)$ by $\left(1-2\nu,\frac{a+b}{2},b\right).$\\
On the other hand, if $\frac{1}{2}\leq\nu\leq 1,$ we have
\begin{equation}
f\left((1-\nu)a+\nu b\right)\left(\frac{f(a)f(b)}{f^2(\frac{a+b}{2})}\right)^{1-A_1(\nu)}
\geq f^{1-\nu}(a)f^{\nu}(b)\prod_{j=1}^{N}\left(\frac{f(t_j(\nu))f(u_j(\nu))}{f^2(w_j(\nu))}\right)^{A_j(2-2\nu)},
\end{equation}
where $t_j(\nu), u_j(\nu)$ and $w_j(\nu)$ are obtained from the above $x_j(\nu), y_j(\nu)$ and $z_j(\nu)$ on replacing $(\nu,a,b)$ by $\left(2-2\nu,a,\frac{a+b}{2}\right).$
\end{corollary}

The following is a squared additive version for log-convex functions. This inequality will help prove some squared versions of certain means.
\begin{theorem}\label{squared_log_convex_first}
Let $f:[a,b]\to [0,\infty)$ be log-convex. Then for $0\leq \nu\leq 1$ and $N\geq 2$, we have
\begin{eqnarray*}
f^2((1-\nu)a+\nu b)&+&A_1^2(\nu)\Delta_1f^2(\nu;a,b)+\sum_{j=2}^{N}A_j(\nu)\Delta_jf^2(\nu;a,b)\\
&\leq& \left((1-\nu)f(a)+\nu f(b)\right)^2.
\end{eqnarray*}
\end{theorem}
\begin{proof}
We prove the result for $0\leq \nu\leq \frac{1}{2}.$ Since $f$ is log-convex, it follows that $g=f^2$ is log-convex too, and hence is convex. Therefore, Theorem \ref{first_main_theorem} implies
$$g((1-\nu)a+\nu b)+\sum_{j=1}^{N}A_j(\nu)\Delta_jg(\nu;a,b)\leq (1-\nu)g(a)+\nu g(b),$$ which implies, for $0\leq\nu\leq\frac{1}{2},$
\begin{eqnarray}
\nonumber f^2((1-\nu)a+\nu b)&+&\nu^2\Delta_1f^2(\nu;a,b)+\sum_{j=2}^{N}A_j(\nu)\Delta_jf^2(\nu;a,b)\\
\label{needed_square_first}&\leq&\left((1-\nu)f(a)+\nu f(b)\right)^2+H(\nu;a,b),
\end{eqnarray}
where
\begin{eqnarray*}
H(\nu;a,b)&=&(1-\nu)f^2(a)+\nu f^2(b)+\nu^2 \Delta_1f^2(\nu;a,b)-\nu \Delta_1f^2(\nu;a,b)\\
&&-\left((1-\nu)f(a)+\nu f(b)\right)^2\\
&=&2\nu(1-\nu)\left(f^2\left(\frac{a+b}{2}\right)-f(a)f(b)\right)\\
&\leq&0,
\end{eqnarray*}
where the last inequality follows from log-convexity of $f$. Since $H(\nu;a,b)\leq 0$, it follows from (\ref{needed_square_first}) that
\begin{eqnarray*}
 f^2((1-\nu)a+\nu b)&+&\nu^2\Delta_1f^2(\nu;a,b)+\sum_{j=2}^{N}A_j(\nu)\Delta_jf^2(\nu;a,b)\\
&\leq&\left((1-\nu)f(a)+\nu f(b)\right)^2.
\end{eqnarray*}
Similar computations imply the result for $\frac{1}{2}\leq\nu\leq 1.$
\end{proof}
Then reversed squared versions maybe obtained in a similar way from Theorem \ref{first_reversed_theorem} as follows.
\begin{theorem}\label{reversed_square_theorem}
Let $f:[a,b]\to[0,\infty)$ be log-convex and $N\in\mathbb{N}$. If $0\leq\nu\leq \frac{1}{2},$ we have
\begin{eqnarray*}
f^2((1-\nu)a+\nu b)&+&(1-\nu)^2\Delta_1 f^2(\nu;a,b)+2\nu(1-\nu)\left(f(a)f(b)-f^2\left(\frac{a+b}{2}\right)\right)\\
&\geq&\left((1-\nu)f(a)+\nu f(b)\right)^2\\
&&+\sum_{j=1}^{N}A_j(1-2\nu)\Delta_jf^2\left(1-2\nu;\frac{a+b}{2},b\right).
\end{eqnarray*}
If $\frac{1}{2}\leq\nu\leq 1,$ we have
\begin{eqnarray*}
f^2((1-\nu)a+\nu b)&+&\nu^2\Delta_1 f^2(\nu;a,b)+2\nu(1-\nu)\left(f(a)f(b)-f^2\left(\frac{a+b}{2}\right)\right)\\
&\geq&\left((1-\nu)f(a)+\nu f(b)\right)^2\\
&&+\sum_{j=1}^{N}A_j(2-2\nu)\Delta_jf^2\left(2-2\nu;a,\frac{a+b}{2}\right).
\end{eqnarray*}
\end{theorem}

\section{Application}
\subsection{Refinements of means inequalities} In this section we present some interesting applications of the above inequalities. The first result is the following refinement of Young's inequality.
\begin{corollary}
Let $x,y>0, N\in\mathbb{N}$ and $0\leq \nu\leq 1.$ Then
\begin{equation}\label{young}
x\#_{\nu}y+\sum_{j=1}^{N}A_j(\nu)\left(\sqrt[2^j]{y^{k_j(\nu)}x^{2^{j-1}-k_j(\nu)}}-\sqrt[2^j]{y^{k_j(\nu)+1}x^{2^{j-1}-k_j(\nu)-1}}\right)^2\leq x\nabla_{\nu}y.
\end{equation}
\end{corollary}
\begin{proof}
This follows from Theorem \ref{first_main_theorem}, on letting $f(t)=x^{1-t}y^{t}, a=0, b=1.$ Then $f$ is convex. Moreover, direct computations show that $$\Delta_jf(\nu;0,1)=\left(\sqrt[2^j]{y^{k_j(\nu)}x^{2^{j-1}-k_j(\nu)}}-\sqrt[2^j]{y^{k_j(\nu)+1}x^{2^{j-1}-k_j(\nu)-1}}\right)^2.$$
\end{proof}
The above theorem has been recently proved in \cite{sabjmaa} as a refinement of Young's inequality. This inequality refines the corresponding refinements appearing in \cite{kittanehmanasreh} and \cite{zhao}, where the inequality was proved only for $N=1,2.$

On the other hand, letting $f(t)=x!_{t}y$, the weighted harmonic mean, we obtain the following refinement of the arithmetic-harmonic mean inequality.
\begin{corollary}\label{arith_har_cor}
Let $x,y>0, N\in\mathbb{N}$ and $0\leq \nu\leq 1.$ Then
\begin{equation}\label{arith_har_ineq}
x!_{\nu}y+\sum_{j=1}^{N}A_j(\nu)\left(x!_{\alpha_j(\nu)}y+x!_{\beta_j(\nu)}y-2x!_{\gamma_j(\nu)}y\right)\leq x\nabla_{\nu}y,
\end{equation}
where $\alpha_j(\nu)=\frac{[2^{j-1}\nu]}{2^{j-1}}, \beta_j(\nu)=\frac{[2^{j-1}\nu]+1}{2^{j-1}}$ and $\gamma_j(\nu)=\frac{\alpha_j(\nu)+\beta_j(\nu)}{2}.$
\end{corollary}
This inequality is a significant refinement of the corresponding inequality in \cite{zuo}, where the inequality was proved only for $N=1.$

Now noting log-convexity of the function $t\mapsto x!_{t}y$  on $[0,1],$ and applying Corollary \ref{corollar_log_convex_first}, we get the following multiplicative refinement of the geometric-harmonic mean inequality.
\begin{corollary}\label{mult_arith_har}
Let $x,y>0, N\in\mathbb{N}$ and $0\leq \nu\leq 1$, we have
$$x!_{\nu}y\prod_{j=1}^{N}\left(\frac{(x!_{\alpha_j(\nu)}y)(x!_{\beta_j(\nu)}y)}{(x!_{\gamma_j(\nu)}y)^2}\right)^{A_j(\nu)}\leq x\#_{\nu}y,$$
where $\alpha_j(\nu)=\frac{[2^{j-1}\nu]}{2^{j-1}}, \beta_j(\nu)=\frac{[2^{j-1}\nu]+1}{2^{j-1}}$ and $\gamma_j(\nu)=\frac{\alpha_j(\nu)+\beta_j(\nu)}{2}.$
\end{corollary}
When $N=1$, Corollary \ref{mult_arith_har}, reduces to $$(x!_{\nu}y)\left(\frac{x\nabla y}{x\#y}\right)^{2\nu}\leq x\#_{\nu}y, 0\leq \nu\leq \frac{1}{2}$$ and
$$(x!_{\nu}y)\left(\frac{x\nabla y}{x\#y}\right)^{2(1-\nu)}\leq x\#_{\nu}y, \frac{1}{2}\leq \nu\leq 1.$$ The constant $\left(\frac{x\nabla y}{x\#y}\right)^{2}$ appearing in these inequalities is called the Kantorovich constant, and has appeared in recent refinements of these mean inequalities. One can see \cite{liao} as a recent reference treating some inequalities using this constant.

As for the squared version, applying Theorem \ref{squared_log_convex_first} to the log-convex functions $t\mapsto x\#_t y$ and $t\mapsto x!_t y$ implies the following. The first inequality refines the corresponding results in \cite{omarkittaneh} and \cite{zhao}, while the other inequality is new.
\begin{corollary}
Let $x,y>0, 0\leq \nu\leq 1, N\geq 2$ and
$$\alpha_j(\nu)=\frac{k_j(\nu)}{2^{j-1}}, \beta_j(\nu)=\frac{k_j(\nu)+1}{2^{j-1}}\;{\text{and}}\;\gamma_j(\nu)=\frac{\alpha_j(\nu)+\beta_j(\nu)}{2}.$$ Then
\begin{eqnarray*}
\left( x\#_{\nu}y \right)^2&+&A_1^2(\nu)(x-y)^2+\sum_{j=2}^{N}A_j(\nu)\left(x^{1-\alpha_j(\nu)}y^{\alpha_j(\nu)}
-x^{1-\beta_j(\nu)}y^{\beta_j(\nu)}\right)^2\\
&\leq&(x\nabla_{\nu}y)^2,
\end{eqnarray*}
and
\begin{eqnarray*}
\left(x!_{\nu}y\right)^2&+&2A_1^2(\nu)(x^2\nabla y^2-(x!y)^2)+\sum_{j=2}^{N}A_j(\nu)\left((x!_{\alpha_j(\nu)}y)^2+(x!_{\beta_j(\nu)}y)^2-2(x!_{\gamma_j(\nu)}y)^2\right)\\
&\leq&(x\nabla_{\nu}y)^2.
\end{eqnarray*}
\end{corollary}
\subsection{Reversed Version}
Applying Theorem \ref{first_reversed_theorem} to the function $f(t)=x\#_{t}y$ implies the following reversed version of Young's inequality.
\begin{corollary}\label{reversed_young_cor}
For $x,y>0$, let $f(t)=x\#_{t}y$ and let $N\in\mathbb{N}$. If $0\leq\nu\leq \frac{1}{2},$ we have
\begin{eqnarray*}
x\#_{\nu}y+(1-\nu)(\sqrt{x}-\sqrt{y})^2\geq x\nabla_{\nu}y+\sum_{j=1}^{N}A_j(1-2\nu)\Delta_jf\left(1-2\nu;\frac{1}{2},1\right).
\end{eqnarray*}
On the other hand, if $\frac{1}{2}\leq\nu\leq 1,$ we have
\begin{eqnarray*}
x\#_{\nu}y+\nu(\sqrt{x}-\sqrt{y})^2\geq x\nabla_{\nu}y+\sum_{j=1}^{N}A_j(2-2\nu)\Delta_jf\left(2-2\nu;0,\frac{1}{2}\right).
\end{eqnarray*}
\end{corollary}
These inequalities refine those in \cite{kittanehmanasreh} and \cite{zhao}.
Then an arithmetic-harmonic reversed version maybe obtained by applying Theorem \ref{first_reversed_theorem} to the function $f(t)=x!_{t}y$ as follows.

\begin{corollary}\label{reversed_arith_har_cor}
For $x,y>0$, let $f(t)=x\#_{t}y$ and let $N\in\mathbb{N}$. If $0\leq\nu\leq \frac{1}{2},$ we have
\begin{eqnarray*}
x!_{\nu}y+(1-\nu)\left(x+y-2x!y\right)\geq x\nabla_{\nu}y+\sum_{j=1}^{N}A_j(1-2\nu)\Delta_jf\left(1-2\nu;\frac{1}{2},1\right).
\end{eqnarray*}
On the other hand, if $\frac{1}{2}\leq\nu\leq 1,$ we have
\begin{eqnarray*}
x!_{\nu}y+\nu(x+y-2 x!y)\geq x\nabla_{\nu}y+\sum_{j=1}^{N}A_j(2-2\nu)\Delta_jf\left(2-2\nu;0,\frac{1}{2}\right).
\end{eqnarray*}
\end{corollary}
These inequalities refine those in \cite{liao}.

Similarly, noting log-convexity of the function $f(t)=x!_ty$, we may apply Corollary \ref{corollar_log_convex_second} to obtain reversed multiplicative version of the harmonic-geometric mean inequality. We leave the application to the reader.

Following the same guideline, we may obtain reversed squared versions by applying Theorem \ref{reversed_square_theorem} to the functions $t\mapsto x\#_{t}y$ and $x\mapsto x!_{t}y.$ Observe that when $f(t)=x\#_{t}y$ we have $f(a)f(b)-f^2\left(\frac{a+b}{2}\right)=0.$ Therefore, applying Theorem \ref{reversed_square_theorem} implies the following inequalities, which refine the corresponding inequalities in \cite{omarkittaneh} and \cite{zhao}.
\begin{corollary}
Let $x,y>0$ and $N\geq 1.$ If $0\leq\nu\leq 1,$ we have
\begin{eqnarray*}
\left(x\#_{\nu}y\right)^2+(1-\nu)^2(x-y)^2\geq (x\nabla_{\nu}y)^2+\sum_{j=1}^{N}A_j(1-2\nu)\Delta_jf^2\left(1-2\nu;\frac{1}{2},1\right).
\end{eqnarray*}
If $\frac{1}{2}\leq\nu\leq 1,$ we have
\begin{eqnarray*}
\left(x\#_{\nu}y\right)^2+\nu^2(x-y)^2\geq (x\nabla_{\nu}y)^2+\sum_{j=1}^{N}A_j(2-2\nu)\Delta_jf^2\left(1-2\nu;0,\frac{1}{2}\right).
\end{eqnarray*}
\end{corollary}

Now letting $g(t)=x!_{t}y$ we obtain the following new inequalities for the arithmetic-harmonic means.

\begin{corollary}
Let $x,y>0$ and $N\geq 1.$ If $0\leq\nu\leq 1,$ we have
\begin{eqnarray*}
\left(x!_{\nu}y\right)^2&+&2(1-\nu)^2(x^2\nabla y^2-(x!y)^2)+2\nu(1-\nu)\left(xy-\left(\frac{2xy}{x+y}\right)^2\right)\\
&\geq& (x\nabla_{\nu}y)^2+\sum_{j=1}^{N}A_j(1-2\nu)\Delta_jg^2\left(1-2\nu;\frac{1}{2},1\right).
\end{eqnarray*}
If $\frac{1}{2}\leq\nu\leq 1,$ we have
\begin{eqnarray*}
\left(x!_{\nu}y\right)^2&+&2\nu^2(x^2\nabla y^2-(x!y)^2)+2\nu(1-\nu)\left(xy-\left(\frac{2xy}{x+y}\right)^2\right)\\
&\geq& (x\nabla_{\nu}y)^2+\sum_{j=1}^{N}A_j(2-2\nu)\Delta_jg^2\left(1-2\nu;0,\frac{1}{2}\right).
\end{eqnarray*}
\end{corollary}

\subsection{Some $L^p$ inequalities}
Let $(X,\mathcal{M},\mu)$ be a measure space, and let $0<p<q<r.$ Then $L^p\cap L^r\subset L^q$ and
$$\|f\|_q\leq \|f\|_p^{\nu}\|f\|_r^{1-\nu},\;{\text{where}}\;f\in L^p\cap L^r\;{\text{and}}\;\nu=\frac{q^{-1}-r^{-1}}{p^{-1}-r^{-1}}.$$
This inequality can be modified using Corollary \ref{corollar_log_convex_first} and a reversed version can be obtained using Corollary \ref{corollar_log_convex_second}.
\begin{proposition}\label{first_prop_Lp}
Let $(X,\mathcal{M},\mu)$ be a measure space,  $0<p<q<r$ and $\nu$ be as above. If $f\in L^p\cap L^r$ and $N\in\mathbb{N}$, then  we have
$$\|f\|_q\prod_{j=1}^{N}\left(\frac{\|f\|_{x_j^{-1}(\nu)}\|f\|_{y_j^{-1}(\nu)}}{\|f\|^2_{z_j^{-1}(\nu)}}\right)^{A_j(\nu)}\leq \|f\|_p^{\nu}\|f\|_r^{1-\nu}.$$
\end{proposition}
\begin{proof}
It is easy to check that the function $h(t)=\|f\|_{1/t}$ is log-convex on $[r^{-1},p^{-1}]$. Then direct application of Corollary \ref{corollar_log_convex_first} implies the result.
\end{proof}
In particular, when $N=1$, the above proposition implies
$$\|f\|_q\leq \left\{\begin{array}{cc}\|f\|_r^{1-2\nu}\|f\|^{2\nu}_{\frac{2pr}{p+r}},&0\leq \nu\leq \frac{1}{2}\\
\|f\|_p^{2\nu-1}\|f\|^{2-2\nu}_{\frac{2pr}{p+r}},&\frac{1}{2}\leq \nu\leq 1\end{array}\right\}\leq \|f\|_p^{\nu}\|f\|_r^{1-\nu}.$$
The condition $0\leq\nu\leq\frac{1}{2}$ can be interpreted as $\frac{2pr}{p+r}\leq q\leq r$, while $\frac{1}{2}\leq\nu\leq 1$ means $p\leq q\leq\frac{2pr}{p+r}.$

Moreover, a reversed version maybe obtained using Corollary \ref{corollar_log_convex_second}.
\begin{proposition}\label{second_prop_Lp}
Let $(X,\mathcal{M},\mu)$ be a measure space,  $0<p<q<r$ and $\nu=\frac{q^{-1}-r^{-1}}{p^{-1}-r^{-1}}$. If $\frac{2pr}{p+r}\leq q\leq r$, $f\in L^p\cap L^r$ and $N\in\mathbb{N}$, then
$$\|f\|_q\geq \|f\|^{2-2\nu}_{\frac{2pr}{p+r}}\|f\|_{p}^{2\nu-1}\prod_{j=1}^{N}\left(  \frac{\|f\|_{t_j^{-1}(\nu)}\|f\|_{u_j^{-1}(\nu)}}{\|f\|^2_{w_j^{-1}(\nu)}}         \right)^{A_j(1-2\nu)}\geq \|f\|^{2-2\nu}_{\frac{2pr}{p+r}}\|f\|_{p}^{2\nu-1},$$
where  $t_j,u_j$ and $z_j$ are obtained from $x_j,y_j$ and $z_j$ by replacing $(\nu,a,b)$ with $\left(1-2\nu,\frac{p+r}{2pr},p^{-1}\right).$ On the other hand, if $p\leq q\leq\frac{2pr}{p+r},$ then
$$\|f\|_q\geq \|f\|^{2\nu}_{\frac{2pr}{p+r}}\|f\|_{r}^{1-2\nu}\prod_{j=1}^{N}\left(  \frac{\|f\|_{t_j^{-1}(\nu)}\|f\|_{u_j^{-1}(\nu)}}{\|f\|^2_{w_j^{-1}(\nu)}}         \right)^{A_j(1-2\nu)}\geq \|f\|^{2\nu}_{\frac{2pr}{p+r}}\|f\|_{r}^{1-2\nu},$$
where  $t_j,u_j$ and $z_j$ are obtained from $x_j,y_j$ and $z_j$ by replacing $(\nu,a,b)$ with $\left(2-2\nu,r^{-1},\frac{p+r}{2pr}\right).$
\end{proposition}
Propositions \ref{first_prop_Lp} and \ref{second_prop_Lp} have been obtained using log-convexity of the function $h(t)=\|f\|_{t^{-1}}.$ In fact, noting log-convexity of the function $h(t)=\|f\|_t^{t}$, we obtain the same results! This is due to the equivalence of log-convexity of the functions $t\mapsto \|f\|_{t^{-1}}$ and that of $\|f\|_t^t.$ We refer the reader to \cite{saboam} where these relations between the different log-convex function criteria have been discussed.

The celebrated three lines lemma of Hadamard states the following.
\begin{lemma}\label{three_lines_lemma}
Let $\mathbb{D}=\{z\in\mathbb{C}:0\leq \Re z\leq 1\}$ and let $\varphi:\mathbb{D}\to\mathbb{C}$ be continuous on $\mathbb{D}$ and analytic in the interior of $\mathbb{D}$. Then the function $f:[0,1]\to\mathbb{R}$ defined by $f(x)=\sup_{y}|\varphi(x+iy)|$ is log-convex.
\end{lemma}
This lemma is an extremely useful tool in the theory of complex functions. In particular, this lemma becomes handy in proving different interpolation versions of bounded linear operators between $L^p$ spaces.

Log-convexity implied by Lemma \ref{three_lines_lemma} allows us to apply our refined and reversed versions for log-convex functions. In the following proposition, we present one term refinement and reverse.
\begin{proposition}
Let $\mathbb{D}=\{z\in\mathbb{C}:0\leq \Re z\leq 1\}$ and let $\varphi:\mathbb{D}\to\mathbb{C}$ be continuous on $\mathbb{D}$ and analytic in the interior of $\mathbb{D}$. Then the function $f:[0,1]\to\mathbb{R}$ defined by $f(x)=\sup_{y}|\varphi(x+iy)|$ satisfies the following
$$f(x)\leq\left\{\begin{array}{cc} f^{1-2x}(0)f^{2x}\left(\frac{1}{2}\right),&0\leq x\leq\frac{1}{2}\\ f^{2x-1}(1)f^{2-2x}\left(\frac{1}{2}\right),&\frac{1}{2}\leq x\leq 1\end{array}\right.$$ and
$$f(x)\geq\left\{\begin{array}{cc} f^{2x-1}(1)f^{2-2x}\left(\frac{1}{2}\right),&0\leq x\leq\frac{1}{2}\\ f^{1-2x}(0)f^{2x}\left(\frac{1}{2}\right),&\frac{1}{2}\leq x\leq 1\end{array}\right..$$
\end{proposition}

\subsection{Operator versions}
The following theorem provides a refinement of the well known Heinz inequality and its reverse. The proof follows immediately noting convexity of the Heinz means, see \cite{B2}.
\begin{theorem}
For $A,B\in\mathbb{M}_n^+, X\in\mathbb{M}_n, 0\leq \nu\leq 1$ and any unitarily invariant norm $\||\;\;\||$, let
$$f(\nu)=\||A^{\nu}XB^{1-\nu}+A^{1-\nu}XB^{\nu}\||.$$
Then we have the following refinement of Heinz inequality
\begin{eqnarray*}
\||A^{\nu}XB^{1-\nu}+A^{1-\nu}XB^{\nu}\||+\sum_{j=1}^{N}A_j(\nu)\Delta_jf(\nu;0,1)\leq \||AX+XB\||.
\end{eqnarray*}
Moreover, if $0\leq\nu\leq \frac{1}{2},$ we have
\begin{eqnarray*}
\||A^{\nu}XB^{1-\nu}+A^{1-\nu}XB^{\nu}\||&+&2(1-\nu)\left(\||AX+XB\||-\||\sqrt{A}X\sqrt{B}\||\right)\\
&\geq&\||AX+XB\||+\sum_{j=1}^{N}A_j(1-2\nu)\Delta_jf\left(1-2\nu;\frac{1}{2},1\right).
\end{eqnarray*}
On the other hand, if $\frac{1}{2}\leq\nu\leq 1,$ we have
\begin{eqnarray*}
\||A^{\nu}XB^{1-\nu}+A^{1-\nu}XB^{\nu}\||&+&2\nu\left(\||AX+XB\||-\||\sqrt{A}X\sqrt{B}\||\right)\\
&\geq&\||AX+XB\||+\sum_{j=1}^{N}A_j(2-2\nu)\Delta_jf\left(2-2\nu;0,\frac{1}{2}\right).
\end{eqnarray*}
\end{theorem}
In \cite{saboam}, it is shown that for $A,B\in\mathbb{M}_n^+$ and $X\in\mathbb{M}_n,$ the functions
$$t\to \||A^{t}XB^{1-t}\||, t\to\||A^{t}XB^{1-t}\||\;\||A^{1-t}X^B{t}\||, t\to \tr(A^{t}XB^{1-t}X^*)$$
are log-convex on $[0,1].$ Therefore, we may apply Corollaries \ref{corollar_log_convex_first} and \ref{corollar_log_convex_second} to obtain refinements and reversed versions for such functions.\\
For the $\|\;\;\|_2$ norm, we can prove log convexity of the Heinz means, which allows us to obtain further refinements of the Heinz inequality by applying
Corollaries \ref{corollar_log_convex_first} and \ref{corollar_log_convex_second}.
\begin{proposition}
Let $A,B\in\mathbb{M}_n^{+}$ and $X\in\mathbb{M}_n$, and define $f(\nu)=\|A^{\nu}XB^{1-\nu}+A^{1-\nu}XB^{\nu}\|_2.$ Then $f$ is log-convex on $[0,1]$.
\end{proposition}
\begin{proof}
Since $A,B\in\mathbb{M}_n^+$, there are diagonal matrices $D_1:={\text{diag}}(\lambda_i), D_2:={\text{diag}}(\mu_i)$ and unitarily matrices $U,V$ such that $\lambda_i,\mu_i\geq 0,$ $A=UD_1U^*$ and $B=VD_2V^*.$ Letting $Y=U^*XV,$ we have
$$A^{\nu}XB^{1-\nu}+A^{1-\nu}XB^{\nu}=U(\lambda_i^{\nu} y_{ij}\mu_j^{1-\nu}+\lambda_i^{1-\nu} y_{ij}\mu_j^{\nu})V^*.$$ Since $\|\;\;\|_2$ is a unitarily invariant norm, we have
\begin{eqnarray*}
f^2(\nu)&=&\|U(\lambda_i^{\nu} y_{ij}\mu_j^{1-\nu}+\lambda_i^{1-\nu} y_{ij}\mu_j^{\nu})V^*\|_2^2\\
&=&\sum_{i,j}\left(\lambda_i^{\nu}\mu_j^{1-\nu}+\lambda_i^{1-\nu}\mu_j^{\nu}\right)^2|y_{ij}|^2.
\end{eqnarray*}
Notice that each summand is log-convex, being the square of a log-convex function. This implies that $f^2$ is log-convex. Consequently, $f$ is log-convex.
\end{proof}
Letting $f(\nu)=\|A^{\nu}XB^{1-\nu}+A^{1-\nu}XB^{\nu}\|_2$ and applying Theorem \ref{squared_log_convex_first} imply the following squared version of Heinz inequality.
\begin{corollary}
Let $A,B\in\mathbb{M}_n^{+}, X\in\mathbb{M}_n, 0\leq \nu\leq 1$ and $N\geq 2.$ Then
\begin{eqnarray*}
&&\|A^{\nu}XB^{1-\nu}+A^{1-\nu}XB^{\nu}\|_2^2+2A_1^2(\nu)\left(\|AX+XB\|_2^2-2\|A^{\frac{1}{2}}XB^{\frac{1}{2}}\|_2^2\right)\\
&&\hspace{0.3cm}+\sum_{j=2}^{N}A_j(\nu)\Delta_jf^2(\nu;0,1)\leq \|AX+XB\|_2^2.
\end{eqnarray*}
\end{corollary}

We leave the application of Corollaries \ref{corollar_log_convex_first} and \ref{corollar_log_convex_second} to the reader.

Further operator versions maybe obtained using Lemma \ref{monotone}. The following operator versions refine the corresponding results in \cite{kittanehmanasreh} and \cite{zhao}.
\begin{proposition}
Let $A,B\in\mathbb{M}_n^{++}$ and $0\leq \nu\leq 1.$ Then for $\alpha_j(\nu)=\frac{k_j(\nu)}{2^{j-1}}, \beta_j(\nu)=\frac{k_j(\nu)+1}{2^{j-1}}, \gamma_j(\nu)=\frac{\alpha_j(\nu)+\beta_j(\nu)}{2}$ and $N\in\mathbb{N}$, we have
\begin{eqnarray*}
A\#_{\nu}B+\sum_{j=1}^{N}A_j(\nu)\left(A\#_{\alpha_j(\nu)}+A\#_{\beta_j(\nu)}-2A\#_{\gamma_j(\nu)}B\right)\leq A\nabla_{\nu}B.
\end{eqnarray*}
\end{proposition}
\begin{proof}
In Corollary \ref{young}, let $x=1$, expand the summand and apply Lemma \ref{monotone} with $y$ replaced by $X=A^{-\frac{1}{2}}BA^{-\frac{1}{2}}.$ Then the result follows upon conjugating both sides with $A^{\frac{1}{2}}.$
\end{proof}
In a similar way one may obtain reversed versions by applying Corollary \ref{reversed_young_cor}. This provides refinements of the reversed versions of \cite{zhao}. The following is an operator arithmetic-harmonic version, refining the corresponding results in \cite{zuo}.
\begin{proposition}
Let $A,B\in\mathbb{M}_n^{++}$ and $0\leq\nu\leq 1.$ Then for $\alpha_j(\nu)=\frac{k_j(\nu)}{2^{j-1}}, \beta_j(\nu)=\frac{k_j(\nu)+1}{2^{j-1}}, \gamma_j(\nu)=\frac{\alpha_j(\nu)+\beta_j(\nu)}{2}$ and $N\in\mathbb{N}$, we have
\begin{eqnarray*}
A!_{\nu}B+\sum_{j=1}^{N}A_j(\nu)\left(A!_{\alpha_j(\nu)}+A!_{\beta_j(\nu)}-2A!_{\gamma_j(\nu)}B\right)\leq A\nabla_{\nu}B.
\end{eqnarray*}
\end{proposition}
The proof follows immediately on applying Lemma \ref{monotone} together with Corollary \ref{arith_har_cor}.
On the other hand, applying Corollary \ref{reversed_arith_har_cor} implies the following refinement of the corresponding inequalities in \cite{liao}.
\begin{proposition}
Let $A,B\in\mathbb{M}_n^{++}$ and $N\in\mathbb{N}$. If $0\leq\nu\leq 1,$ we have
\begin{eqnarray*}
A!_{\nu}B&+&(1-\nu)(A+B-2A!B)\\
&\geq& A\nabla_{\nu}B+\sum_{j=1}^{N}A_j(1-2\nu)\left(A!_{\alpha_j(\nu)}B+A!_{\beta_j(\nu)}B-2A!_{\gamma_j(\nu)}B\right),
\end{eqnarray*}
where $\alpha_j(\nu)=\frac{1}{2}\left(1-\frac{k_j(1-2\nu)}{2^{j-1}}\right)+\frac{k_j(1-2\nu)}{2^{j-1}},$ $\beta_j(\nu)=\frac{1}{2}\left(1-\frac{k_j(1-2\nu)+1}{2^{j-1}}\right)+\frac{k_j(1-2\nu)+1}{2^{j-1}}$ and $\gamma_j(\nu)=\frac{\alpha_j(\nu)+\beta_j(\nu)}{2}.$\\
On the other hand, if $\frac{1}{2}\leq\nu\leq 1,$ we have
\begin{eqnarray*}
A!_{\nu}B&+&\nu(A+B-2A!B)\\
&\geq& A\nabla_{\nu}B+\sum_{j=1}^{N}A_j(2-2\nu)\left(A!_{\alpha_j(\nu)}B+A!_{\beta_j(\nu)}B-2A!_{\gamma_j(\nu)}B\right),
\end{eqnarray*}
where $\alpha_j(\nu)=\frac{k_j(2-2\nu)}{2^{j-1}},$ $\beta_j(\nu)=\frac{k_j(2-2\nu)+1}{2^{j-1}}$ and $\gamma_j(\nu)=\frac{\alpha_j(\nu)+\beta_j(\nu)}{2}.$
\end{proposition}
The following is an interesting one-term multiplicative  refinement of the operator geometric-harmonic mean inequality.
\begin{theorem}
Let $A,B\in\mathbb{M}_n^{++}$ and $0\leq \nu\leq 1.$ Then
\begin{eqnarray*}
(A!_{\nu}B)\left(\frac{A^{-1}B+2I+B^{-1}A}{4}\right)^{r}\leq A\#_{\nu}B,
\end{eqnarray*}
where $r=\min\{\nu,1-\nu\}.$
\end{theorem}
\begin{proof}
We prove the desired inequality for $0\leq \nu\leq \frac{1}{2}.$ In Corollary \ref{mult_arith_har}, let $N=1$ and $x=1$, to get
$(1!_{\nu}y)\left(\frac{1+y}{2\sqrt{y}}\right)^{2\nu}\leq 1\#_{\nu}y,$ or
\begin{eqnarray}\label{needed_int_har_geo}
\frac{1}{4^{\nu}}\left((1-\nu)+\nu y^{-1}\right)^{-1}\left(y+2+y^{-1}\right)^{\nu}\leq y^{\nu}.
\end{eqnarray}
Let $X=A^{-\frac{1}{2}}BA^{-\frac{1}{2}}$ and apply Lemma \ref{monotone}. The left hand side of (\ref{needed_int_har_geo}) becomes
\begin{eqnarray}
\nonumber&&\frac{1}{4^{\nu}}\left((1-\nu)I+\nu A^{\frac{1}{2}}B^{-1}A^{\frac{1}{2}}\right)^{-1}\left(A^{-\frac{1}{2}}BA^{-\frac{1}{2}}+2I+A^{\frac{1}{2}}B^{-1}A^{\frac{1}{2}}\right)^{\nu}\\
\nonumber&=&\frac{1}{4^{\nu}}\left[A^{-\frac{1}{2}}(A!_{\nu}B)A^{-\frac{1}{2}}\right]\left[A^{\frac{1}{2}}\left(A^{-1}B+2I+B^{-1}A\right)^{\nu}A^{-\frac{1}{2}}\right]\\
\label{needed_second_int_har_geo}&=&A^{-\frac{1}{2}}(A!_{\nu}B)\left(\frac{A^{-1}B+2I+B^{-1}A}{4}\right)^{\nu}A^{-\frac{1}{2}}.
\end{eqnarray}
On the other hand,  the right hand side of (\ref{needed_int_har_geo}) is simply $\left(A^{-\frac{1}{2}}BA^{-\frac{1}{2}}\right)^{\nu}.$ This together with (\ref{needed_second_int_har_geo}) imply the desired inequality, upon conjugating both sides with $A^{\frac{1}{2}}.$ This completes the proof.
\end{proof}

\end{document}